\documentclass[11pt]{amsart}
\usepackage[left=1in,right=1in,top=1in,bottom=1in]{geometry}

\usepackage[all,cmtip]{xy}
\usepackage{amssymb}
\usepackage[english]{babel}
\usepackage{mathrsfs}

\newtheorem{theorem}{Theorem}[section]
\newtheorem{lemma}[theorem]{Lemma}

\newtheorem{proposition}[theorem]{Proposition}
\newtheorem{corollary}[theorem]{Corollary}

\theoremstyle{definition}
\newtheorem{definition}[theorem]{Definition}

\newtheorem{remark}[theorem]{Remark}

\newtheorem{question}[theorem]{Question}

\newtheorem{claim}{Claim}

\def\B{\mathcal{B}}

\def\N{\mathbb{N}}

\def\Z{\mathbb{Z}}

\def\map{{minimally almost periodic}}

\def\Cyc{\mathrm{Cyc}}
\def\grp#1{\langle{#1}\rangle}
\def\ssgp{$\mathrm{SSGP}^*$}
\def\Prm{\mathbf{P}}
\def\Z{\mathbb{Z}}
\def\Q{\mathbb{Q}}
\def\P{\mathbb{P}}
\def\F{\mathbb{F}}
\def\a{m}

\def\ssgp{\mathrm{SSGP}}
\def\SSGP{\mathtt{SSGP}}

\def \Zet[#1]{\lceil #1 \rceil}  

\def\Oplus{+}

\begin{document}
\title[SSGP topologies on abelian groups]
{SSGP 
topologies
on abelian groups of positive finite divisible rank}

\author[D. Shakhmatov]{Dmitri Shakhmatov}
\address{Division of Mathematics, Physics and Earth Sciences\\
Graduate School of Science and Engineering\\
Ehime University, Matsuyama 790-8577, Japan}
\email{dmitri.shakhmatov@ehime-u.ac.jp}
\thanks{The first listed author was partially supported by the Grant-in-Aid for Scientific Research~(C) No.~26400091 of the Japan Society for the Promotion of Science (JSPS)}

\author[V. Ya\~nez]{V\'{\i}ctor Hugo Ya\~nez}
\address{Master's Course, Graduate School of Science and Engineering\\
Ehime University, 
Matsuyama 790-8577, Japan}
\email{victor\textunderscore yanez@comunidad.unam.mx}
\thanks{This paper was written as part of the second listed author's 
studies
at the Graduate School of Science and Engineering of Ehime University. The second listed author was partially supported by the 2016 fiscal year grant of the Matsuyama Saibikai.}

\begin{abstract}
For a subset $A$ of a group $G$, we denote by $\grp{A}$ the smallest subgroup of $G$ containing $A$ and let $\Cyc(A)=\{x\in G:
\grp{x}\subseteq A\}$.
A
topological group $G$ 
is {\em SSGP\/}
if 
$\grp{\Cyc(U)}$ is dense in $G$ for every neighbourhood $U$ of zero of $G$. SSGP groups form a proper subclass of the class of minimally almost periodic groups.

Comfort and Gould asked for a characterization of abelian groups 
which admit an SSGP 
group 
topology.
An ``almost complete'' characterization
was found by Dikranjan and the first author. The remaining case is resolved here.
As a corollary, we give a positive answer to another question of Comfort and Gould by showing that 
if an abelian group admits an SSGP($n$) 
group topology for some positive integer $n$,
then
it admits an SSGP group topology as well.
\end{abstract}

\dedicatory{In memory of W. Wistar Comfort}



\maketitle

As usual, $\Z$ and $\Q$ denote the groups of integer numbers and rational numbers respectively, 
$\N$ denotes the set of natural numbers and $\N^+=\N\setminus\{0\}$. We use $\Prm$ to denote the set of prime numbers.

Let $G$ be a group.
For subsets $A,B$ of 
$G$, we let $AB=\{ab:a\in A, b\in B\}$ and $A^{-1}=\{a^{-1}:a\in A\}$. When $G$ is abelian, we use the additive notation $A+B$ instead of $AB$ and $-A$ instead of $A^{-1}$.
For a subset $A$ of 
$G$, we
denote by $\grp{A}$ the smallest subgroup of $G$ containing $A$.
To simplify the notation, 
we write $\grp{x}$ instead of $\grp{\{x\}}$ for $x\in G$.

\section{The small subgroup generating property $\ssgp$} 

Following \cite{DS_SSGP}, we define 
\begin{equation}
\label{eq:Cyc}
\Cyc(A)=\{x\in G:
\grp{x}\subseteq A\}
\ 
\text{ for every }
A\subseteq G.
\end{equation}

\begin{definition}
\label{SSGP:original}
A topological group $G$ has the {\em small subgroup generating property\/} (abbreviated to {\em SSGP\/})
if and only if 
$\grp{\Cyc(U)}$ is dense in $G$ for every neighbourhood $U$ of the identity of 
$G$.
We shall say that a topological group $G$ is {\em SSGP\/} if $G$ satisfies the small subgroup generating property.
\end{definition}

\begin{remark}
The
small subgroup generating property was introduced by Gould \cite{G}; see also \cite{Gould, CG}. 
He 
says that a topological group $G$ has the small subgroup generating property if for every neighbourhood  $U$ of the identity of $G$, there exists  a family  $\mathcal H$ of subgroups of $G$ such that $\bigcup \mathcal H\subseteq U$ and 
$\grp{\bigcup \mathcal H}$
is dense in $G$.
It is easily seen that Gould's definition is equivalent to the one given 
in~\ref{SSGP:original}.
\end{remark}

\begin{definition}
\label{def:bounded:envelope}
For a subset $A$ of 
group 
$G$
and $k \in \N^+$,
we let
\begin{equation}
\label{eq:bounded:hull}
\langle A \rangle_k= \left \{ \prod_{i=1}^j a_i: j\le k, a_1,\dots,a_j\in A \right \}.
\end{equation}
\end{definition} 

\begin{remark}
\label{remark}
If $A$ is a subset of a group $G$ such that 
$A=A^{-1}$,
then $\grp{A} = \bigcup_{k \in \N^+} \grp{A}_k$.
\end{remark}

Using the notation from \eqref{eq:bounded:hull},
one can get 
a convenient 
reformulation of 
the $\ssgp$ 
property.

\begin{proposition}
\label{SSGP:reformulation}
For a topological group $G$, the following conditions are equivalent:
\begin{itemize}
\item[(i)]
$G$  has the small subgroup generating property $\ssgp$;
\item[(ii)]
$G = \bigcup_{k \in \N^+} W \cdot \langle \Cyc(W) \rangle_k$ for every neighbourhood $W$ of the identity 
of $G$;
\item[(iii)] for each $g\in G$ and every neighbourhood $W$ of the identity of $G$, there exist $k\in\N^+$ (depending on $g$ and $W$), $g_0\in W$  and $g_1,\dots,g_k\in G$ such that 
$g=\prod_{i=0}^k g_i$
and $\grp{g_i}\subseteq W$ for $i=1,\dots,k$.
\end{itemize}
\end{proposition}
\begin{proof}
(i)$\to$(ii)
Suppose that $G$ has the small subgroup generating property.
Let $W$ be a neighbourhood of the identity $e_G$ of $G$.
Since $\Cyc(W)=(\Cyc(W))^{-1}$, we have
$\grp{\Cyc(W)}=\bigcup_{k\in\N^+} \grp{\Cyc(W)}_k$ by Remark \ref{remark}.
Furthermore, $\grp{\Cyc(W)}$ is dense in $G$ by Definition \ref{SSGP:original}, so
$$G=W \cdot\grp{\Cyc(W)} = W\cdot \bigcup_{k\in\N^+} \grp{\Cyc(W)}_k = \bigcup_{k\in\N^+} W \cdot \grp{\Cyc(W)}_k.
$$

(ii)$\to$(i) 
Suppose that $G$ satisfies the property from item (ii) of our lemma. Fix a neighbourhood $U$ of $e_G$ in $G$.
By Definition \ref{SSGP:original}, 
to check that $G$ has the small subgroup generating property,
it suffices to prove that $\grp{\Cyc(U)}$ is dense in $G$.
This is equivalent to establishing
the equality $G=V \cdot \grp{\Cyc(U)}$ for every neighbourhood $V$ of $e_G$ in $G$. Let $V$ be an arbitrary neighbourhood of $e_G$ in $G$.
Then $W=U\cap V$ is a neighbourhood of $e_G$ in $G$, so we can apply (ii) to this $W$ to obtain 
$$G = \bigcup_{k\in\N^+} W \cdot \grp{\Cyc(W)}_k
=
W\cdot \bigcup_{k\in\N^+}  \grp{\Cyc(W)}_k
 \subseteq V \cdot \bigcup_{k\in\N^+} \grp{\Cyc(U)}_k \subseteq V \cdot \grp{\Cyc(U)}.$$
The converse inclusion $V \cdot \grp{\Cyc(U)} \subseteq G$ is clear.

(ii)$\leftrightarrow$(iii) follows from
\eqref{eq:Cyc} and \eqref{eq:bounded:hull}. 
\end{proof}

\section{Properties $\ssgp(n)$ and $\SSGP(\alpha)$} 

According to von Neumann's terminology \cite{vN}, a  topological group $G$ is called
{\em minimally almost periodic\/} if every continuous homomorphism from $G $ to a compact group is trivial. 

For every $n\in\N$, 
the property $\ssgp(n)$ 
of topological groups was
defined in \cite{G}. We refer the reader to \cite[Definition 3.3]{CG}
for the definition of $\ssgp(n)$.
It was proved in \cite[Remark 3.4, Theorem 3.5]{CG} that
the following implications hold for an arbitrary topological group:
\begin{equation}\label{SSGP}
\ssgp =  \ssgp(1) \to \ssgp(2) \to \ldots \to  \ssgp(n) \to \ldots \to  \mbox{ \map}.
\end{equation}

Examples distinguishing all 
properties
in (\ref{SSGP}) can be found in  \cite[Corollary 3.14, Theorem 4.6]{CG}.
However,
all properties
in (\ref{SSGP}) coincide for bounded torsion abelian topological groups. 

\begin{theorem}  
\cite[Corollary 2.23]{CG} \label{CoroCG2}
A bounded torsion abelian topological group has the small subgroup generating property $\ssgp$ if and only if it is \map. 
\end{theorem}  

For every ordinal $\alpha$, the property $\SSGP(\alpha)$ of topological groups was 
defined in 
\cite[Definition 2.3]{DS_SSGP}.
The reader is alerted to a subtle difference in fonts used for distinguishing properties $\ssgp(n)$ and $\SSGP(n)$, for $n\in\N$.
The connection between these properties is given in items (i), (iv) and (v) of the following theorem.
\begin{theorem}
\label{SSGP:implications}
\begin{itemize}
\item[(i)]
$\SSGP(1)$ coincides with $\ssgp$
(which is equivalent to $\ssgp(1)$ by \eqref{SSGP});
\item[(ii)] $\SSGP(\alpha)\to\SSGP(\beta)$ whenever $\alpha,\beta$ are ordinals satisfying $\alpha<\beta$;
\item[(iii)] 
for each ordinal $\alpha$, 
$\SSGP(\alpha) \to\mbox{\map}$;
\item[(iv)]
$\ssgp(n)\to \SSGP(n)$ for every $n\in\N^+$;
\item[(v)]
for every 
$n\in \N^+$,
an abelian
 topological group has the property $\ssgp(n)$ 
if and only if 
it has the property
$\SSGP(n)$.
\end{itemize}
\end{theorem}
\begin{proof}
Item (i) is \cite[Lemma 2.2 and Equation 7]{DS_SSGP},
item (ii) is \cite[Proposition 5.1]{DS_SSGP},  item (iii) is
\cite[Proposition 5.3~(iii)]{DS_SSGP},
item (iv) is
\cite[Corollary 6.3]{DS_SSGP}
and
item (v) is
\cite[Theorem 6.4]{DS_SSGP}.
\end{proof}

It is unknown if the properties $\ssgp(n)$ and $\SSGP(n)$
coincide for arbitrary topological groups \cite[Question 13.4]{DS_SSGP}.

\section{Introduction}

\begin{definition}
We shall say that a topology $\mathscr{T}$ on a group $G$ is an {\em SSGP topology\/} provided that $(G,\mathscr{T})$ is a topological group with the property $\ssgp$. Similar terminology will be used for properties $\ssgp(n)$ for $n\in\N^+$ and $\SSGP(\alpha)$ for an ordinal $\alpha$.
\end{definition}

A characterization of abelian groups admitting a minimal almost periodic group topology is obtained in \cite{DS_MinAP}.

Comfort and Gould \cite[Questions 5.2 and 5.3]{CG} asked the following question.

\begin{question}\label{ques1} 
\begin{itemize}
 \item[(a)] 
 What are the (abelian) groups 
 which admit an $\ssgp$ topology?
 \item[(b)] 
 Does every abelian group 
which for some $n > 1$ admits an $\ssgp(n)$ topology 
 also admit an $\ssgp$ topology?
\end{itemize}
\end{question}

Following \cite{Fuchs}, for an abelian group $G$, we denote by $r_0(G)$ the free rank of $G$, by $r_p(G)$ the $p$-rank of $G$ for a prime number $p$, and we let 
\begin{equation*}
r(G) = \max\left\{r_0(G),\ \sum\{r_p(G):p\in \Prm \}\right\},
\end{equation*}
where $\Prm$ denotes the set of prime numbers.

\begin{definition}\label{divisible:rank} 
\cite[Definition 7.2]{DS_ConnectedMarkov}  
For an abelian group $G$, 
the cardinal
\begin{equation}
\label{r_d(G)}
r_d(G)=\min\{r(nG):n\in\N^+\}
\end{equation}
is called
the {\em divisible rank\/} of $G$. 
\end{definition}

The notion of the divisible rank was defined, under the name of {\em final rank\/}, by Szele \cite{S} for $p$-groups. 

\begin{remark}
\label{rem:rank:0}
An
abelian group $G$ satisfies $r_d(G)=0$ if and only if $G$ is a bounded torsion group;  that is, if $nG=\{0\}$ for some $n\in\N^+$.
\end{remark}

In view of Remark \ref{rem:rank:0}, 
the next theorem coincides with 
\cite[Corollary 1.7]{DS_SSGP}.

\begin{theorem}
\label{bounded:case}
A non-trivial abelian group $G$ satisfying $r_d(G)=0$ admits an $\ssgp$ 
topology if and only if all leading Ulm-Kaplanski invariants of $G$ are infinite. 
\end{theorem}

Theorem \ref{bounded:case} resolves Question \ref{ques1}~(a)  for abelian groups of divisible rank zero,
while the next theorem resolves it 
for abelian groups of infinite divisible rank.
\begin{theorem}
\label{infinite:rank}
\cite[Theorem 3.2]{DS_SSGP}
Every abelian group $G$ satisfying $r_d(G)\ge\omega$
  admits an $\ssgp$ 
topology.
\end{theorem}

In \cite{DS_SSGP}, Dikranjan and the first author reduced 
the remaining case $0< r_d(G)<\omega$
to a
question regarding the existence of $\ssgp$ topologies on abelian groups of a very
particular type. 
More precisely, assuming a positive answer to their question,
formulated in \cite{DS_SSGP} as Question 13.1, they established in \cite[Theorem 13.2]{DS_SSGP} what is stated
below as Theorem \ref{unconditional:corollary}, thus giving a ({\em provisional\/}, at that moment) complete characterization of
abelian groups $G$ admitting an $\ssgp$ topology in the remaining open case $0< r_d(G)<\omega$.

\section{Results}

The main goal of the paper is to provide a positive answer to
(a more general version of) \cite[Question 13.1]{DS_SSGP} 
given in
Theorem \ref{the:theorem} below. 
It follows from this theorem 
that 
\cite[Theorem 13.2]{DS_SSGP}
holds {\em unconditionally\/}; that is, the following result holds:

\begin{theorem}
\label{unconditional:corollary}
For an abelian group $G$ satisfying $1\le r_d(G)<\omega$,
the following conditions are equivalent:
\begin{itemize}
\item[(i)] $G$ admits an $\ssgp$ topology;
\item[(ii)] $G$ admits an $\SSGP(\alpha)$ topology for some ordinal $\alpha$;
\item[(iii)] the quotient $H=G/t(G)$ of $G$ with respect to its torsion part $t(G)$ has finite rank $r_0(H)$ and $r(H/A)=\omega$ for some (equivalently, every) free 
subgroup $A$ of $H$
such that $H/A$ is torsion.  
\end{itemize}
\end{theorem}

A combination of Theorems \ref{bounded:case},
\ref{infinite:rank} and 
\ref{unconditional:corollary}
provides a complete solution to item (a) of 
Question \ref{ques1} for abelian groups.
Similarly,
the implication (iii)$\to$(i) of the next corollary
provides 
a complete solution to item (b) of the same question for abelian groups. Both items of Question \ref{ques1}  remain widely open for non-commutative groups.

\begin{corollary}
\label{main:corollary}
For an abelian group $G$, the following conditions are equivalent:
\begin{itemize}
\item[(i)] $G$ admits an $\ssgp$ topology;
\item[(ii)] $G$ admits an $\SSGP(\alpha)$ topology for some ordinal $\alpha$;
\item[(iii)] $G$ admits an $\ssgp(n)$ topology for some integer $n>1$.
\end{itemize}
\end{corollary}
\begin{proof}
The implication (i)$\to$(iii)
follows from 
\eqref{SSGP}, 
while the 
implication (iii)$\to$(ii)
follows from item (iv)
of
Theorem \ref{SSGP:implications}.

(ii)$\to$(i) Let $G$ be an abelian $\SSGP(\alpha)$ group for some ordinal $\alpha$.
We consider three cases.

\smallskip
{\em Case 1\/}.
$r_d(G)=0$. 
By Remark \ref{rem:rank:0},
 $G$ is a bounded torsion group.
Since $G$ is an $\SSGP(\alpha)$ group, it is minimally almost periodic by item (iii) of
Theorem \ref{SSGP:implications}, 
so
$G$ 
is $\ssgp$ 
by Theorem \ref{CoroCG2}. 

\smallskip
{\em Case 2\/}.
$1\le r_d(G)<\omega$. 
In this case $G$ admits an $\ssgp$ topology by the implication (ii)$\to$(i) of 
Theorem
\ref{unconditional:corollary}.

\smallskip
{\em Case 3\/}.
$r_d(G)\ge\omega$. In this case $G$ admits an $\ssgp$ topology by Theorem \ref{infinite:rank}.
\end{proof}

The paper is organized as follows.
In Section \ref{Sec:3} we prove two lemmas about the algebraic structure of subgroups $\Q_\pi^\a$ of finite powers
$\Q^\a$ of $\Q$ which are needed in the proof of the main theorem.
In Section \ref{Sec:4} we introduce the notion of a wide subgroup $G$ of $\Q^\a$
and we establish two auxiliary
lemmas about wide subgroups needed in the sequel.
The main result of the paper is 
Theorem \ref{the:theorem} which states 
that a direct sum $G\oplus H$ of a wide subgroup $G$ of $\Q^\a$
and at most countable abelian group $H$ admits 
a metric $\ssgp$ topology $\mathscr{T}$.
Let us explain 
the strategy for constructing 
this topology.
First, we 
introduce a poset $(\P,\le)$ in Section 
\ref{Sec:5}.
Next, we define a countable family $\mathscr{D}$ of dense subsets of $(\P,\le)$ and select a linearly ordered 
subset $\F$ of $(\P,\le)$ intersecting all members of the family $\mathscr{D}$; this can be done due to a folklore Lemma \ref{countable:generic:filter}. Finally, the countable base of neighbourhoods of zero for $\mathscr{T}$ is defined by means of 
elements of $\F$ in Section \ref{Sec:7}.
The family $\mathscr{D}$ is defined in \eqref{family:D}.
Section \ref{Sec:6} collects lemmas establishing density 
in $(\P,\le)$ of various sets participating in the family $\mathscr{D}$.
The ``technical heart'' of this section is Lemma \ref{ssgp:density}
which is responsible for the $\ssgp$ property of $\mathscr{T}$.

The reader familiar with Martin's Axiom undoubtedly notices 
that we are using here a ``ZFC version'' of this axiom 
when the family of dense sets is at most countable.
The choice of such an exposition was determined by the author's desire 
to replace a direct construction of $\mathscr{T}$ via an induction
(which would be totally incomprehensible) by a  ``much smoother'' forcing-type argument using a poset $(\P,\le)$ and some dense subsets of it
(which 
is much easier to follow than the direct inductive construction). We hope that, after discovering the 
technical complexity even of this ``smooth'' approach, the reader would fully agree with our judgment.

\section{The algebraic structure of subgroups $\Q_\pi^\a$ of $\Q^\a$}
\label{Sec:3}

\begin{definition}
\label{def:Q:pi}
For a non-empty 
set
$\pi$ of 
of prime numbers, we use $\Q_\pi$ to denote the 
set of all rational numbers $q$ whose irreducible representation 
$q=z/n$ with $z\in\Z$ and $n\in\N^+$ is such that 
all prime divisors of $n$ belong to $\pi$.
For convenience, we let $\Q_{\emptyset}=\{0\}$.
\end{definition}

\begin{definition}
Given an integer $s \in \Z$, we shall denote by $\Zet[s]$ the subgroup $s\Z$ of $\Q$. 
\end{definition}

A straightforward proof of the following lemma is left to the reader.

\begin{lemma}
\label{basic:subgroups:of:Q}
Let $\pi\subseteq \Prm$ and $\a\in\N^+$. Then:
\begin{itemize}
\item[(i)]
$\Q_\pi$ is a subgroup of $\Q$, so $\Q_\pi^\a$
is a subgroup of $\Q^\a$.
\item[(ii)]
$\Z\subseteq \Q_\pi$, and so 
$\Zet[s]^\a \subseteq \Q_\pi^\a$
for each $s\in \N^+$.
\item[(iii)]
If $\pi\subseteq\pi'\subseteq \Prm$, then
$\Q_\pi\subseteq \Q_{\pi'}$, and so $\Q_\pi^\a\subseteq \Q_{\pi'}^\a$.
\end{itemize}
\end{lemma}

Our next lemma clarifies the algebraic structure of subgroups $\Q_\pi^\a$ of $\Q^\a$.

\begin{lemma}
\label{iterative:lemma}
{\rm (A)} Suppose that 
$s\in\Z$, 
$k\in\N^+$,
$g_1,\dots,g_k\in \Q^\a$, 
$\pi_{0},\pi_{1},\dots,\pi_k\in [\Prm]^{<\omega}$,
\begin{equation}
\label{eq:pi}
\pi_0\subseteq \pi_1\subseteq \pi_2\subseteq\dots\subseteq \pi_k,
\end{equation}
and
the following conditions hold for every $j=1,2,\dots,k$:
\begin{itemize}
\item[(a$_j$)] $g_j\in \Q^\a_{\pi_{j}}$,
\item[(b$_j$)] $\grp{g_j}\cap \Q^\a_{\pi_{j-1}}\subseteq
    \Zet[s]^\a$.
\end{itemize}
Then:
\begin{itemize}
\item[(i)] $\grp{\{g_1,\dots,g_i\}}+\Q^\a_{\pi_0}\subseteq \Q^\a_{\pi_i}$ for every $i=1,\dots,k$.
\item[(ii)]  $\grp{\{g_i,\dots,g_k\}}\cap \Q^\a_{\pi_{i-1}}\subseteq \Zet[s]^\a$
for every $i=1,\dots,k$.
\end{itemize}

{\rm (B)}
 In addition to the assumptions of item (A), suppose that
the following condition holds for every $j=1,2,\dots,k$:
\begin{itemize}
\item[(c$_j$)]
$mg_j\notin \Q^\a_{\pi_{j-1}}$ for every $m\in\Z\setminus \{0\}$
satisfying $|m|\le k$.
\end{itemize}
Finally, assume that
$g\in \Q^\a_{\pi_0}$,
$J$ is a proper subset of the set $\{1,2,\dots,k\}$,
$l\in\Z$, $|l|\le k$
and 
\begin{equation}
\label{eq:8:new}
lg_0\in \grp{\{g_j:j\in J\}}+\Q^\a_{\pi_0},
\end{equation}
where
\begin{equation}
\label{eq:g_0}
g_0=g-\sum_{j=1}^k g_j.
 \end{equation} 
Then $l=0$.
\end{lemma}
\begin{proof}
First, we check item (A).

(i) Let $i=1,2,\dots,k$ be arbitrary.
Fix $j=1,\dots,i$.
Then 
$\pi_j\subseteq \pi_i$ by 
\eqref{eq:pi}.
Applying Lemma
\ref{basic:subgroups:of:Q}~(iii),
we get 
$\Q_{\pi_j}^\a\subseteq \Q_{\pi_{i}}^\a$. 
Since $g_j\in \Q_{\pi_j}^\a$ by (a$_j$), 
we obtain 
$g_j\in \Q_{\pi_{i}}^\a$.
Since this holds for arbitrary $j=1,\dots,i$,
we conclude that $\{g_1,\dots,g_i\}\subseteq \Q_{\pi_{i}}^\a$.
Since $\pi_0\subseteq \pi_i$ by \eqref{eq:pi}, 
applying Lemma
\ref{basic:subgroups:of:Q}~(iii) once again,
we obtain that $\Q^\a_{\pi_0}\subseteq \Q^\a_{\pi_i}$.
Since $\Q^\a_{\pi_i}$ is a subgroup of $\Q^\a$ by 
Lemma
\ref{basic:subgroups:of:Q}~(i),
this implies $\grp{\{g_1,\dots,g_i\}}+\Q^\a_{\pi_0}\subseteq \Q^\a_{\pi_i}$.

(ii) We shall prove item (ii) using induction on $k$.

\smallskip
{\em Basis of induction\/}.
If $k=1$,
then
the conclusion of item (ii) holds by (b$_1$).

\smallskip
{\em Inductive step\/}. 
Let $k\ge 2$ and suppose that item (ii) has already been proved for $k-1$.

Let $i=1,\dots,k$ by arbitrary.
Fix 
\begin{equation}
\label{g:in:intersection}
h\in \grp{\{g_i,\dots,g_k\}}\cap \Q^\a_{\pi_{i-1}}.
\end{equation}
Then there exist $x\in \grp{\{g_i,\dots,g_{k-1}\}}$
and
$y\in\grp{g_k}$
such that $h=x+y$.
In particular,
$x\in \Q_{\pi_{k-1}}^\a$
by item (i), as $0\in\Q^\a_{\pi_0}$. 
Since
$i-1\le k-1$, we have $\pi_{i-1}\subseteq \pi_{k-1}$ by 
\eqref{eq:pi},
which implies
$\Q^\a_{\pi_{i-1}}\subseteq \Q^\a_{\pi_{k-1}}$ by 
Lemma
\ref{basic:subgroups:of:Q}~(iii).
From this and \eqref{g:in:intersection}, we obtain $h\in \Q^\a_{\pi_{k-1}}$.
Therefore,
$
y=h-x\in \Q^\a_{\pi_{k-1}}-\Q^\a_{\pi_{k-1}}=\Q^\a_{\pi_{k-1}},
$
as $\Q^\a_{\pi_{k-1}}$
is a subgroup of $\Q^\a$.
Now
$y\in \grp{g_k}\cap \Q^\a_{\pi_{k-1}}\subseteq
\Zet[s]^\a$
by (b$_k$).
Recalling item (ii) of Lemma
\ref{basic:subgroups:of:Q},
we conclude that 
$y\in \Q^\a_{\pi_{i-1}}$.
Since $h\in\Q^\a_{\pi_i}$ by \eqref{g:in:intersection}, 
we get
$x=h-y\in \Q^\a_{\pi_{i-1}}-\Q^\a_{\pi_{i-1}}=\Q^\a_{\pi_{i-1}}$,
as  $\Q^\a_{\pi_{i-1}}$ is a subgroup of $\Q^\a$.
We obtained that
$x\in \grp{\{g_i,\dots,g_{k-1}\}}\cap \Q^\a_{\pi_{i-1}}$.
By our inductive assumption,
$\grp{\{g_i,\dots,g_{k-1}\}}\cap \Q^\a_{\pi_{i-1}}\subseteq \Zet[s]^\a$,
so $x\in \Zet[s]^\a$.
Since $y\in \Zet[s]^\a$ as well,
$h=x+y\in \Zet[s]^\a+\Zet[s]^\a=\Zet[s]^\a$.

Next, we check item (B).
From \eqref{eq:8:new} and \eqref{eq:g_0},
 we get 
$$
lg-\sum_{j=1}^k l g_j\in \grp{\{g_j:j\in J\}}+\Q^\a_{\pi_0}.
$$
Since 
$\Q^\a_{\pi_0}$ is a subgroup of $\Q^\a$ by Lemma
\ref{basic:subgroups:of:Q}~(i)
and
$g\in \Q^\a_{\pi_0}$ by our assumption, it follows that
\begin{equation}
\label{eq:new:approach}
-\sum_{j=1}^k l g_j\in \grp{\{g_j:j\in J\}}+\Q^\a_{\pi_0}.
\end{equation}
Since $J$ is a proper subset of the set $\{1,2,\dots,k\}$, we can fix 
$i\in \{1,2,\dots,k\}\setminus J$.
From this and \eqref{eq:new:approach},
we can find 
\begin{equation}
\label{eq:xy}
x\in \grp{\{g_1,\dots,g_{i-1}\}}
\ \text{ and }\ 
y\in \grp{\{g_{i+1},\dots,g_{k}\}}
\end{equation}
such that 
\begin{equation}
\label{eq:g_t}
l g_i\in x+y+\Q^\a_{\pi_0}.
\end{equation}
(If $i=1$, we define $x=0$, and if $i=k$, we define $y=0$.)
Then
$$
y\in lg_i-x+\Q^\a_{\pi_0}
\in\grp{g_i}-\grp{\{g_1,\dots,g_{i-1}\}}+\Q^\a_{\pi_0}
=
\grp{\{g_1,\dots,g_{i}\}}+\Q^\a_{\pi_0}
\subseteq 
\Q^\a_{\pi_i}
$$
by 
subitem~(i) of item~(A).
Recalling the second inclusion
in \eqref{eq:xy},
we get
$y\in \grp{\{g_{i+1},\dots,g_{k}\}}\cap \Q^\a_{\pi_i}$.
Now applying 
subitem~(ii) of item~(A),
we conclude that 
$y\in \Zet[s]^\a$.
Combining this with Lemma \ref{basic:subgroups:of:Q}~(ii),
we get $y\in\Q^\a_{\pi_{i-1}}$.
From the
first inclusion in \eqref{eq:xy}
and subitem (i) of item (A),
we get $x\in \Q^\a_{\pi_i}$.
Note that  
$\Q^\a_{\pi_0}\subseteq \Q^\a_{\pi_{i-1}}$ by
subitem (i) of item (A),
as $0\in\grp{\{g_1,\dots.g_i\}}$.
Since $\Q^\a_{\pi_{i-1}}$ is a subgroup of $\Q^\a$, from
$x,y\in \Q^\a_{\pi_{i-1}}$,
the inclusion $\Q^\a_{\pi_0}\subseteq \Q^\a_{\pi_{i-1}}$
and \eqref{eq:g_t},
one obtains $lg_i\in \Q^\a_{\pi_{i-1}}$.
Since $l\in\Z$ and $|l|\le k$, applying 
(c$_i$) with $m=l$, we get $l=0$.
\end{proof}

\section{Wide subgroups of $\Q^\a$}
\label{Sec:4}

Our next definition gives a name to subgroups of $\Q^\a$ having the property 
from 
\cite[Question 13.1]{DS_SSGP}.

\begin{definition}
\label{wide:subgroup:definition}
Let $\a\in\N^+$.
We shall call a subgroup $G$ of $\Q^\a$ {\em wide\/}
if $\Z^\a\subseteq G$ and $G\setminus \Q^\a_\pi\not=\emptyset$
for every $\pi\in[\Prm]^{<\omega}$.
\end{definition}

\begin{lemma}
\label{finding:g}
Let $G$ be a wide subgroup of $\Q^\a$,  $\pi\in[\Prm]^{<\omega}$, $k\in\N\setminus\{0\}$ and $s\in\Z\setminus\{0\}$. Then
there exists $g\in G$ 
such that:
\begin{itemize}
\item[(i)]  $\grp{g}\cap \Q^\a_\pi\subseteq \Zet[s]^\a$,
\item[(ii)]
$lg\notin \Q^\a_\pi$ for every $l\in\Z\setminus \{0\}$
satisfying $|l|\le k$.
\end{itemize}
\end{lemma}
\begin{proof}
Let $\varpi$ be the set of all prime numbers not exceeding
$\max\{k,s\}$.
Then $\pi'=\pi\cup\varpi\in [\Prm]^{<\omega}$.
Since $G$ is nice, we can find $h\in G\setminus \Q^\a_{\pi'}$.
Let $h=(h_1,\dots,h_\a)$, where $h_1,\dots,h_\a\in \Q$.
For every $i=1,\dots,\a$, let $h_i=a_i/b_i$ be the irreducible fraction 
with $a_i\in\Z$ and $b_i\in\N^+$.

Since $h\not\in\Q^\a_{\pi'}$,
we can fix $t\in\{1,\dots,\a\}$ such that $h_t\notin \Q_{\pi'}$.
Since $a_t/b_t$ is an irreducible representation of 
$h_t\in\Q\setminus\Q_{\pi'}$, we can fix $p\in\Prm\setminus\pi'$
dividing $b_t$. 
Since $\varpi\subseteq \pi'$ and $p\in\Prm\setminus\pi'$,
we have $p\in\Prm\setminus\varpi$. Since
$\varpi$ includes all prime numbers not exceeding
$\max\{k,s\}$, we conclude that
\begin{equation}
\label{p:is:bigger:than:k:and:s}
p>\max\{k,s\}.
\end{equation}

For every $i=1,\dots,\a$, let 
\begin{equation}
\label{eq:n_i}
n_i=\max\{n\in\N: p^n\text{ divides }b_i\}.
\end{equation}
Then 
\begin{equation}
\label{eq:c_i}
c_i=b_i /p^{n_i}\in\N
\text{ is not divisible by }
p.
\end{equation} 

Note that
\begin{equation}
\label{eq:n_t}
n_t\ge 1,
\end{equation}
as $p$ divides $b_t$ by our choice of $p$.
It follows from $s\in\N$ and 
\eqref{eq:c_i}
that
\begin{equation}
\label{eq:m}
m_0=s c_1 c_2\cdots c_\a\in\N.
\end{equation} 

Since $G$ is a group and $h\in G$, it follows that 
$g=m_0h\in G$.
Note that 
\begin{equation}
\label{eq:g:coord}
g=(m_0h_1,\dots,m_0h_\a).
\end{equation}
We claim that $g$ is the required element of $G$; that is, conditions (i) and (ii) are satisfied.

Fix $i=1,\dots,\a$.
It follows from \eqref{eq:c_i} and \eqref{eq:m}
that
\begin{equation}
\label{eq:mh_i}
m_0h_i=s c_1 c_2\cdots c_\a \frac{a_i}{b_i}
=
s\left(\prod_{j=1, j\not=i}^\a c_j\right) \frac{b_i}{p^{n_i}}\frac{a_i}{b_i}
=
s\left(\prod_{j=1,j\not=i}^\a c_j\right) \frac{a_i}{p^{n_i}}
=
\frac{sd_i}{p^{n_i}},
\end{equation}
where 
\begin{equation}
\label{eq:d_i}
d_i=c_1\cdots c_{i-1} a_i c_{i+1}\cdots c_\a \in\Z.
\end{equation}
\begin{claim}
\label{irreducible:representation}
The right-hand side of \eqref{eq:mh_i}
is the irreducible representation of $m_0h_i$.
\end{claim}
\begin{proof}
If $n_i=0$, then $p^{n_i}=1$, so $sd_i/1$ is the irreducible representation of $m_0h_i$. Suppose now that $n_i\ge 1$.
In this case, it suffices to check that neither $s$ nor $d_i$ is divisible by $p$.
The first statement follows from \eqref{p:is:bigger:than:k:and:s}.
Since $n_i\ge 1$, \eqref{eq:n_i} implies that $p$ divides $b_i$. 
Since $a_i/b_i$ is an irreducible fraction, $p$ does not divide $a_i$.
By \eqref{eq:c_i}, each $c_j$ is not divisible by $p$. 
By \eqref{eq:d_i}, this means that $p$ does not divide 
$d_i$.
\end{proof}

(i) By \eqref{eq:g:coord}, in order to check (i), we need to show that 
$\grp{m_0h_i}\cap \Q_\pi\subseteq s\Z$ for every $i=1,\dots,\a$.
Fix such an integer $i$.
Let $x\in \grp{m_0h_i}\cap \Q_\pi$.
Then $x=lm_0h_i$ for some 
$l\in\Z$.
Since $p\notin\pi$, it follows from $x\in\Q_\pi$, Definition \ref{def:Q:pi} and 
Claim \ref{irreducible:representation}
that $p^{n_i}$ must divide $l$, so
$l=p^{n_i} l'$ for some $l'\in\Z$.
Therefore, 
$x=lm_0h_i=l'sd_i$ by \eqref{eq:mh_i}.
Since $l'\in\Z$ and $d_i\in\Z$ by \eqref{eq:d_i},
we conclude that 
$x\in \Zet[s]$.
This implies $\grp{m_0h_i}\cap \Q_\pi\subseteq \Zet[s]$.

(ii) By \eqref{eq:g:coord}, in order to check (ii), 
it suffices to show that
$lm_0h_t\notin \Q^\a_\pi$ for every $l\in\Z\setminus \{0\}$
satisfying $|l|\le k$.
Fix $l$ satisfying these conditions.
Since $|l|\le k<p$ by 
\eqref{p:is:bigger:than:k:and:s}, 
$p$ does not divide $l$.
Recalling Claim \ref{irreducible:representation}
and \eqref{eq:mh_i}, we conclude that
${lsd_t}/{p^{n_t}}$ is the irreducible representation
of 
$lm_0h_t$.
Since $p\notin\pi$ and $n_t\ge 1$ by \eqref{eq:n_t}, 
from this and Definition \ref{def:Q:pi},
we obtain that $lm_0h_t\not\in\Q_\pi$.
\end{proof}

\begin{lemma}
\label{finding:a:sequence:of:g's}
Let $\a\in\N^+$, $G$ be a wide subgroup of $\Q^\a$,  $\pi_0\in[\Prm]^{<\omega}$, $k\in\N^+$ and $s\in\Z\setminus\{0\}$. 
Then
there exist an increasing sequence $\pi_0\subseteq \pi_1\subseteq \pi_2\subseteq\dots\subseteq \pi_k$
of finite subsets of $\Prm$ and elements
$g_1,\dots,g_k\in G$ 
such that conditions (a$_j$), (b$_j$), (c$_j$) 
of Lemma 
\ref{iterative:lemma} hold for every $j=1,\dots,k$.
\end{lemma}
\begin{proof}
The lemma is proved by finite induction on $k\in\N^+$ based on 
Lemma \ref{finding:g}.
\end{proof}

\section{The poset $(\P,\le)$}
\label{Sec:5}

In this section we define a (technically rather involved) poset $(\P,\le)$ which shall be used to construct a metric $\ssgp$ 
topology $\mathscr{T}$ on the direct sum $G\oplus H$ of a wide subgroup of
$\Q^\a$ for $\a\in\N^+$ and an arbitrary at most countable abelian group $H$. The definition of a poset itself does not require the group $G$ is to be wide and the countability restriction on $H$ is not essential, so we do not impose either of these two conditions in the next definition. 

\begin{definition}
\label{def:P}
Let $H$ be an abelian group. For a fixed $m\in\N^+$, 
consider the direct sum $\Q^\a\oplus H$. 
Furthermore, let $G$ be a subgroup of $\Q^\a$ containing $\Z^\a$.
Then the sum 
$G+H=G\oplus H$ is direct.

\begin{itemize}
\item[(a)] Let 
$\P$ be the set of all structures $p=\ll \pi^p,n^p,\{U_i^p:i\in n^p+1\},\{s_i^p:i\in n^p+1\}\gg$ satisfying the following conditions:
\begin{itemize}
\item[(1$_p$)] $\pi^p\in[\Prm]^{<\omega}$,
\item[(2$_p$)] $n^p\in\N$,
\item[(3$_p$)] $s_i^p\in\N^+$ for every $i\in n^p+1$,
\item[(4$_p$)]  $0 \in U_i^p\subseteq (G\cap \Q^\a_{\pi^p}) \Oplus H$ for every $i\in n^p+1$,
\item[(5$_p$)] $-U_{i}^p= U_i^p$ for every 
$i\in n^p+1$,
\item[(6$_p$)] $U_{i}^p+\Zet[s_i^p]^\a = U_i^p$ for every 
$i\in n^p+1$,
\item[(7$_p$)] $U_{i+1}^p+U_{i+1}^p\subseteq U_i^p$ for every 
$i\in n^p$,
\item[(8$_p$)] $s_i^p$ divides $s_{i+1}^p$ for every 
$i\in n^p$.
\end{itemize}

\item[(b)] Given structures
$$
p=\ll \pi^p,n^p,\{U_i^p:i\in n^p+1\},\{s_i^p:i\in n^p+1\}\gg\in\P
$$
and 
$$q=\ll \pi^q,n^q,\{U_i^q:i\in n^q+1\},\{s_i^q:i\in n^q+1\}\gg\in\P,
$$ we define
$q\le p$ if and only if the following conditions are satisfied:
\begin{itemize}
\item[(i$_q^p$)] $\pi^p\subseteq \pi^q$,
\item[(ii$_q^p$)] $n^p\le n^q$,
\item[(iii$_q^p$)] $U_i^q\cap (\Q^\a_{\pi^p} \Oplus H)=U_i^p$ for every $i\in n^p+1$,
\item[(iv$_q^p$)] $s_i^q=s_i^p$ for every $i\in n^p+1$.
\end{itemize}
\end{itemize}
\end{definition}

\begin{remark}
\label{subgroups:in:U_i^p}
Let
$p=\ll \pi^p,n^p,\{U_i^p:i\in n^p+1\},\{s_i^p:i\in n^p+1\},\{s_i^p:i\in n^p+1\}\gg\in\P$. Then:
\begin{itemize}
\item[(i)]
{\em $U_{i+1}^p\subseteq U_i^p$ for every $i\in n^p$.\/}
Indeed, $0 \in U_{i+1}^p$ by (4$_p$), so
$U_{i+1}^p=0+U_{i+1}^p\subseteq U_{i+1}^p+U_{i+1}^p\subseteq U_i^p$ by (7$_p$).
\item[(ii)] {\em $\Zet[s_i^p]^\a  \subseteq U_i^p$ for every $i\in n^p+1$.\/} Indeed, 
$U^p_{i} + \Zet[s_i^p]^\a = U^p_{i}$ by
(6$_p$). Since $0 \in U^p_{i}$ by (4$_p$),
this implies $\Zet[s_i^p]^\a \subseteq U^p_{i}$. 
\end{itemize}
\end{remark}

A straightforward proof of the following lemma is left to the reader.
\begin{lemma}
$(\P,\le)$ is a 
poset.
\end{lemma}

\begin{lemma}
$\P\not=\emptyset$.
\end{lemma}
\begin{proof}
Define
$\pi^p=\emptyset$,
$n^p=0$,
$s_0^p=1$,
$U_0^p=\Z^\a$, and
$$p=\ll \pi^p,n^p,\{U_i^p:i\in 1\},\{s_i^p:i\in 1\}\gg
=
\ll \emptyset,0,\{U_0^p\},\{s_0^p\}\gg.$$
To show that $p\in\P$, we need to check the conditions 
(1$_p$)--(8$_p$) of Definition \ref{def:P}~(a).
Conditions (1$_p$)--(3$_p$) are clear, and conditions (7$_p$)
and (8$_p$) are vacuous. Conditions (5$_p$)
and (6$_p$) are satisfied, as $U_0^p=\Z^\a$ is a subgroup of $\Q^\a$ and $\Zet[1]^\a=\Z^\a$.
\end{proof}

\begin{lemma}
\label{adding:extra:neighbourhood}
Given $p\in\P$,  
$g\in G\cap \Q_{\pi_p}^\a$ and $h\in H$ with $g+h\neq 0$, 
one can find $q\in\P$ such that
$q\le p$, 
$n^q=n^p+1$
and 
$g+h\notin U_{n^q}$.
\end{lemma}
\begin{proof}
Let 
\begin{itemize}
\item $\pi^q = \pi^p$,
\item 
$n^q = n^p+1$,
\item 
$U^q_i = U^p_i$
and
$s^q_i = s^p_i$
for every $i \in n^p+1$.
\end{itemize}
It remains only to define $U_{n^q}^q$ and $s_{n^q}^q$.

Since $\bigcap_{k\in\N^+} \Zet[ks^p_{n^p}]^\a=\{0\}$ and $g+h\not=0$,
there exists $k\in\N^+$ such that $g+h\not\in\Zet[ks^p_{n^p}]^\a$.
Define $s^q_{n_q}=ks^p_{n^p}$. Since $s_{n^p}^p\in\N^+$ by (3$_p$) and $k\in\N^+$, 
$s^q_{n_q}\in\N^+$ 
and
\begin{equation}
\label{division:eq}
s_{n^p}^p\
 \text{ divides } 
\ 
s_{n^q}^q.
\end{equation}
By our choice of $s^q_{n_q}$, we have
\begin{equation}
\label{eq:g+h}
g+h\notin \Zet[s^q_{n^q}]^\a.
\end{equation}

Finally, we define 
\begin{equation}
\label{eq:def:Un^q}
U_{n^q}^q=\Zet[s^q_{n^q}]^\a.
\end{equation}

\begin{claim}
$q=\ll \pi^q,n^q,\{U_i^q:i\in n^q+1\},\{s_i^q:i\in n^q+1\}\gg\in\P$. 
\end{claim}
\begin{proof}
According to Definition \ref{def:P}~(a), we have to check 
that the structure $q$ satisfies conditions (1$_q$)--(8$_q$).
By our construction, conditions (1$_q$), (2$_q$), (3$_q$) and (8$_q$) hold. 

Let us check conditions (4$_q$)--(7$_q$).
Since conditions (4$_p$)--(7$_p$)
hold and 
$U^q_i = U^p_i$,
$s^q_i = s^p_i$
for every $i \in n^p+1$,
it follows that conditions
(4$_q$)--(6$_q$) hold for each $i \in n^p+1=n^q$
and condition
(7$_q$) 
holds for each $i \in n^p$.
Therefore, it 
remains only to check 
the following four conditions:
\begin{itemize}
\item[(a)]
$0 \in U_{n^q}^q\subseteq (G\cap \Q^\a_{\pi^q}) \Oplus H$, 
\item[(b)]
$-U_{n^q}^q=U_{n^q}^q$, 
\item[(c)]
$U_{n^q}^q+ \Zet[s^q_{n^q}]^\a = U_{n^q}^q$,
\item[(d)] $U_{n_q}^q+U_{n_q}^q\subseteq U_{n_p}^q$.
\end{itemize}
Conditions (b) and (c) are immediate from 
\eqref{eq:def:Un^q}
and the fact that $\Zet[s^q_{n^q}]^\a$ is a subgroup of 
$\Q^\a$.

Clearly, $0\in \Zet[s^q_{n^q}]^\a   \subseteq \Z^\a\subseteq G$
by our assumption on $G$. Furthermore,
$\Zet[s^q_{n^q}]^\a  \subseteq \Q_{\pi^q}^q$ 
by Lemma \ref{basic:subgroups:of:Q}~(ii).
Combining this with 
\eqref{eq:def:Un^q}, we get 
$0\in U_{n_q}^q\subseteq G\cap \Q^\a_{\pi^q}\subseteq (G\cap \Q^\a_{\pi^q}) \Oplus H$.
Thus, (a) holds.

From \eqref{eq:def:Un^q}, we get 
$U_{n_q}^q+U_{n_q}^q= \Zet[s^q_{n^q}]^\a + \Zet[s^q_{n^q}]^\a=
\Zet[s^q_{n^q}]^\a$.
Since $s_{n^p}^p$ divides $s_{n^q}^q$ by \eqref{division:eq}, 
we have the inclusion
$\Zet[s^q_{n^q}]^\a \subseteq \Zet[s^p_{n^p}]^\a$.
Finally, $\Zet[s^p_{n^p}]^\a \subseteq U_{n_p}^q$
by Remark
\ref{subgroups:in:U_i^p}~(ii).
This finishes the check of (d).
\end{proof}

The inequality $q\le p$ is clear from our construction of $q$ and Definition \ref{def:P}~(b).
Finally, from 
\eqref{eq:g+h}
and
\eqref{eq:def:Un^q},
we get 
$g+h\notin U_{n^q}$.
\end{proof}

\section{Density lemmas}
\label{Sec:6}

\begin{definition}
Let $(\P,\le)$ be a poset.
Recall that a set $D\subseteq \P$ is called:
\begin{itemize}
\item[(i)] 
   {\em dense in $(\P,\le)$\/} provided that for every $p\in \P$ there exists $q\in D$ such that $q\le p$;
\item[(ii)]   
   {\em downward-closed in $(\P,\le)$\/} if for every $p \in D$ and $q \in \P$ the inequality $q \leq p$ implies that $q \in D$.
\end{itemize}    
\end{definition}

The relation between these two notions is made apparent by the following straightforward lemma.

\begin{lemma}
\label{downwardclosed:dense:sets}
If $A,B \subseteq \P$ are dense subsets of a poset $(\P,\le)$ and $A$ is downward-closed, then $A \cap B$ is dense in $(\P,\le)$.
\end{lemma}

\begin{lemma}
\label{trivial:dense:sets:i}
For every $n\in\N$, the set $A_n=\{q\in\P: n\le n^q\}$ is dense and downward-closed in $(\P,\le)$.
\end{lemma}
\begin{proof}
Let 
$n\in\N$ and $p\in\P$. 
If $n\le n^p$, then $p\in A_n$, so we shall assume from now on 
that $n^p < n$. 
Then $k=n-n^p\ge 1$.

Note that 
$\Z^\a\subseteq G$ by our assumption on $G$
and $\Z^\a\subseteq \Q_{\pi^p}^\a$ by Lemma
\ref{basic:subgroups:of:Q}~(ii), so
$\Z^\a\subseteq G\cap \Q_{\pi^p}^\a$.
This allows us to fix $g\in G\cap \Q_{\pi^p}^\a$
with $g\neq 0$.
Let $h=0$. Then $g+h=g\not=0$.

Let $q_0=p$. 
By finite induction on $i=1,\dots,k$, we can use 
Lemma \ref{adding:extra:neighbourhood} to find
$q_i\in\P$ such that $q_i\le q_{i-1}$ and $n^{q_i}=n^{q_{i-1}}+1=n^p+i$. (Note that $g$ and $h$ play a ``dummy role'' in this argument; their sole purpose here is to make 
the assumptions of Lemma \ref{adding:extra:neighbourhood} satisfied.)
Now $q_k\le q_{k-1}\le\cdots\le q_1\le q_0=p$ and
$n^{q_k}=n^p+k=n$, so $q_k\in A_n$.
This shows that $A_n$ is dense in $(\P,\le)$.

Finally, given $p \in A_n$ and $q \in \P$ such that $q \leq p$, we have that $n \leq n^p \leq n^q$, thus showing that $q \in A_n$ and therefore, that $A_n$ is downward-closed.
\end{proof}

\begin{lemma}
\label{trivial:dense:sets:ii}
For every $\pi\in[\Prm]^{<\omega}$, the set $B_\pi=\{q\in\P: \pi\subseteq \pi^q\}$ is dense
in $(\P,\le)$.
\end{lemma}
\begin{proof}
Let
$\pi \in [\Prm]^{<\omega}$ and $p\in\P$. Define

\begin{itemize}
\item $\pi^q = \pi^p \cup \pi$,
\item $n^q = n^p$,
\item $s^q_i = s^p_i$ for every $i \in n^p+1$; and 
\item $U^q_i = U^p_i$ for every $i \in n^p+1$.
\end{itemize}

A straightforward check using Definition \ref{def:P}~(a) shows that
$$
q=\ll
 \pi^q,n^q,\{U_i^q:i\in n^q+1\},\{s_i^q:i\in n^q+1\}
\gg\in \P. 
$$
From our definition of $q$ and Definition \ref{def:P}~(b), one easily 
concludes that $q\le p$.
Clearly, $q\in B_\pi$.
This shows that
$B_\pi$ is dense in $(\P,\le)$. 
\end{proof}

\begin{lemma}
\label{trivial:dense:sets:iii}
The set
$C_{g+h}=\{q\in \P: g+h \in (\Q_{\pi^q}^\a \Oplus H )\setminus U^q_{n^q}\}$  is dense in $(\P,\le)$
whenever $g\in G$, $h\in H$ and $g+h\not=0$.
\end{lemma}
\begin{proof}
Fix $g\in G$ and $h\in H$ such that $g+h\not=0$.
Let $r\in\P$ be arbitrary.
Since $g\in G\subseteq \Q^\a$, there exists 
$\pi\in[\Prm]^{<\omega}$ such that
$g\in\Q_{\pi}^\a$. 
Since $B_\pi$ is dense in $(\P,\le)$ 
by Lemma \ref{trivial:dense:sets:ii}, 
there exists $p\in B_\pi$ such that $p\le r$. 
Now $\pi\subseteq \pi^p$ by definition of $B_\pi$,
so $\Q_{\pi}^\a\subseteq \Q_{\pi^p}^\a$
by
Lemma \ref{basic:subgroups:of:Q}~(iii).
Since $g\in\Q_{\pi}^\a$, we also have 
$g\in \Q_{\pi^p}^\a$.
We have checked that
$g\in G\cap \Q_{\pi^p}^\a$.
Applying Lemma \ref{adding:extra:neighbourhood}, we can 
find $q\in \P$ such that $q\le p\le r$
and
$g+h\notin U_{n^q}$.
Since $q\le p$, we have $\pi^p\subseteq \pi^q$ by (i$_q^p$),
so
$\Q_{\pi^p}^\a\subseteq \Q_{\pi^q}^\a$
by
Lemma \ref{basic:subgroups:of:Q}~(iii).
Since $g\in\Q_{\pi^p}^\a$, we also have 
$g\in \Q_{\pi^q}^\a$.
Therefore, $g+h\in \Q_{\pi^q}^\a \Oplus H$.
We have proved that $g+h\in (\Q_{\pi^q}^\a \Oplus H )\setminus U^q_{n^q}$, so $q\in C_{g+h}$.
Since $q\le r$, this shows that
$C_{g+h}$  is dense in $(\P,\le)$.
\end{proof}

\begin{lemma}
\label{ssgp:density}
If $G$ is a wide subgroup of $\Q^\a$, then
 the set 
\begin{equation}
\label{def:D_g}
D_{g+h}=\left\{q\in\P: g+h\in \bigcup_{k \in \N^+} U_{n^q}^q + \langle \Cyc(U_{n^q}^q)\rangle_k\right\}
\end{equation}
 is dense in $(\P,\le)$ for all $g\in G$ and $h\in H$.
 \end{lemma}
\begin{proof}
Let $g\in G$, $h\in H$ and $p\in \P$ be arbitrary. 
By Lemma \ref{trivial:dense:sets:ii},
we may assume, without loss of generality, that 
$g \in \Q_{\pi^p}$.
Define 
\begin{equation}
\label{eq:k}
k=2^{n^p}+1.
\end{equation}

Applying Lemma \ref{finding:a:sequence:of:g's}
to $\pi_0=\pi^p$ and $s=s_{n^p}^p$, 
we can find 
an increasing sequence
\begin{equation}
\label{eq:pi:modified}
\pi^p=\pi_0\subseteq \pi_1\subseteq \pi_2\subseteq\dots\subseteq \pi_k=\pi^q
\end{equation}
of finite subsets of $\Prm$ and elements
$g_1,\dots,g_k\in G$ 
such that 
conditions (a$_j$), (b$_j$), (c$_j$)
of Lemma 
\ref{iterative:lemma} (in which we let $s=s_{n^p}^p$) hold for every $j=1,\dots,k$.

Define $g_0$ as in \eqref{eq:g_0}.
Since $G$ is a subgroup of $\Q$, $g\in\Q_{\pi_0}$ and (a$_j$)
holds for every $j=1,\dots,k$, 
it follows from
\eqref{eq:pi}  and \eqref{eq:g_0} that
\begin{equation}
\label{all:gs}
\{g_j:j\in k+1\}\subseteq \Q^\a_{\pi^q}.
\end{equation}

Let
\begin{equation}
\label{eq:2}
U_{n^p}^q=U_{n^p}^p\cup\left(\left(\{g_0+h\}\cup\{-(g_0+h)\}\cup\bigcup_{j=1}^k \grp{g_j}\right)+ \Zet[s^p_{n^p}]^\a \right).
\end{equation}

By finite reverse induction, we define
\begin{equation}
\label{eq:3}
U_i^q
=
U_i^p\cup
(U_{i+1}^q+U_{i+1}^q+\Zet[s^p_i]^\a) 
\end{equation}
for every $i=n^p-1,\dots,0$.

Let $n^q=n^p$ and $s_i^q=s_i^p$ for every $i\in n^q+1=n^p+1$.

\begin{claim}
\label{claim:1a}
$q=\ll \pi^q,n^q,\{U_i^q:i\in n^q+1\}, \{s_i^q:i\in n^q+1\}\gg\in \P$.
\end{claim}
\begin{proof}
We need to check conditions (1$_q$)--(8$_q$).

Conditions (1$_q$), (2$_q$) and (3$_q$) are trivial.

(4$_q$)
By 
(4$_p$) and \eqref{eq:3},
$0\in U_i^p\subseteq  U_i^q$ for every $i\in n^q+1=n^p+1$.
By (4$_p$), $U_i^p\subseteq (G\cap \Q^\a_{\pi^p}) \Oplus H$ for every $i\in n^q+1=n^p+1$.
Since $G$ is a subgroup of $\Q^\a$, from this, 
$g_1,\dots,g_k\in G$
and 
\eqref{eq:2},
it follows that $U_{n^q}^q=U_{n^p}^q\subseteq G \Oplus H$.
Furthermore,
\eqref{all:gs} and \eqref{eq:2} imply that 
$U_{n^q}^q=U_{n^p}^q\subseteq \Q^\a_{\pi_q} \Oplus H$.
Thus, $U_{n^q}^q\subseteq (G\cap \Q^\a_{\pi^q}) \Oplus H$.
Since $(G\cap \Q^\a_{\pi^q}) \Oplus H$ is a subgroup of $\Q^\a \Oplus H$,
from this, \eqref{eq:3} and finite reverse induction on
$i=n^q, n^q-1,\dots,0$, one concludes that 
 $U_{i}^q\subseteq (G\cap \Q^\a_{\pi^q}) \Oplus H$ for every $i=n^q+1$.
 
(5$_q$)
From
\eqref{eq:2} and (5$_p$), we get
$-U_{n^p}=U_{n^p}$.
Starting with this and using 
\eqref{eq:3}, (5$_p$) and finite reverse induction on
$i=n^q, n^q-1,\dots,0$, one concludes that 
$-U_i^q=U_i^q$
 for every $i=n^q+1$.

(6$_q$)
The equation $U_{n^p}^q+\Zet[s^p_{n^p}]^\a =U_{n^p}^q$
follows from 
\eqref{eq:2} and (6$_p$).
Starting with this and using 
\eqref{eq:3}, (6$_p$) and finite reverse induction on
$i=n^p, n^p-1,\dots,0$, one concludes that 
$U_{i}^q+ \Zet[s^p_i]^\a  =U_{i}^q$
 for every $i=n^p+1$.
 Since $n^q=n^p$ and $s_{i}^q=s_i^p$ for every $i\in n^q+1$,
 this establishes (6$_q$).
 
(7$_q$) follows from \eqref{eq:3}.

(8$_q$) follows from (8$_p$), as  $s_{i}^q=s_i^p$ for every $i\in n^q+1$.
\end{proof}

\begin{claim}
\label{claim:2}
$q\in D_{g+h}$.
\end{claim}
\begin{proof}
It follows from \eqref{eq:2} that 
$g_0+h\in U_{n^q}^q$ and $g_1,\dots,g_k\in \Cyc(U_{n^q}^q)$.
Furthermore,
$g= g_0+ \sum_{i=1}^k g_i$
by \eqref{eq:g_0}. Thus, 
$$g + h =
g_0+ h+ \sum_{i=1}^k g_i
\in U_{n^q}^q + \langle \Cyc(U_{n^q}^q)\rangle_k
$$ 
by 
Definition \ref{def:bounded:envelope},
so $q\in D_{g+h}$ by
\eqref{def:D_g}.
\end{proof}

Our final goal is to prove the inequality 
$q\le p$.
Before establishing this, we need to check two auxiliary facts.

\begin{claim}
\label{claim:1}
The inclusion
\begin{equation}
\label{eq:inlcusion}
U^q_i\subseteq 
\bigcup\left\{U^p_{i}+l (g_0+h)+\sum_{j\in J} \grp{g_j}+\Zet[s^p_i]^\a: l\in\Z, |l|\le 2^{n^q-i},
J\in\{1,2,\dots,k\}^{\le 2^{n^q-i}}\right\}
\end{equation}
holds
for every $i\in n^p+1$.
\end{claim}
\begin{proof}
This is proved by reverse induction on $i=n^p, n^p-1,\dots,0$.
For $i=n^p$ this follows from \eqref{eq:2}.

Let us make an inductive step. Suppose that $i\in n^p$ and we have already 
proved the inclusion \eqref{eq:inlcusion} for $i+1$; that is, we have shown that the inclusion
\begin{equation}
\label{eq:inclusion:i+1}
U^q_{i+1}\subseteq 
\bigcup\left\{U^p_{i+1}+l (g_0+h)+\sum_{j\in J} \grp{g_j}+\Zet[s^p_{i+1}]^\a: l\in\Z, |l|\le 2^{n^q-i-1},
J\in\{1,2,\dots,k\}^{\le 2^{n^q-i-1}}\right\}
\end{equation}
holds. We are going to show that the inclusion \eqref{eq:inlcusion}
holds as well.

Let $x\in U_i^q$. By equation 
\eqref{eq:3}, either $x\in U_i^p$ or $x \in U_{i+1}^q+U_{i+1}^q+ \Zet[s^p_i]^\a$.
In the former case, 
$x=x+0 (g_0+h)+0+0$ is the decomposition 
witnessing that $x$ belongs to the right-hand side of 
\eqref{eq:inlcusion}.
 In the latter case,
 $x= y_0+y_1+s_i^p w$, where $w\in\Z^\a$ and $y_t\in U_{i+1}^q$ for $t=0,1$.
 Applying \eqref{eq:inclusion:i+1}, we can find
 $u_t\in U^p_{i+1}$, $l_t\in\Z$, $J_t\in \{1,2,\dots,k\}^{\le 2^{n^q-i-1}}$, $b_t\in \sum_{j\in J_t} \grp{g_j}$ and $z_t\in\Z^\a$ for $t=0,1$ such that $|l_t|\le 2^{n^q-i-1}$ and
$y_t=u_t+l_t (g_0 + h)+b_t+s_{i+1}^p z_t$.
Since $x=y_0+y_1+s_i^p w$, it follows 
that $x=u+l (g_0+h)+b+s_{i+1}^p z +s_i^p w$,
where $u=u_0+u_1$, $l=l_0+l_1$, $b=b_0+b_1$ and $z=z_0+z_1$.
Since $s_{i}^p$ divides $s_{i+1}^p$ by (8$_p$),
$s_{i+1}^p=s_i^p k_0$ for some $k_0\in\Z$, so
\begin{equation}
\label{eq:x}
x=u+l (g_0+h)+b+s_{i}^p (k_0z+w).
\end{equation}

Since $u_t\in U_{i+1}^p$ for $t=0,1$,
$u=u_0+u_1\in U_{i+1}^p+U_{i+1}^p\subseteq U_i^p$ by (7$_p$).
Since $l_t\in\Z$ and $|l_t|\le 2^{n^q-i-1}$  for $t=0,1$, we have 
$l=l_0+l_1\in\Z$ and 
$$
|l|\le |l_0|+|l_1|\le 2^{n^q-i-1}+2^{n^q-i-1}=2\cdot2^{n^q-i-1}=2^{n^q-i}.
$$
Since $J_t\in \{1,2,\dots,k\}^{\le 2^{n^q-i-1}}$  for $t=0,1$,
the set $J=J_0\cup J_1$ satisfies
$$
|J|\le |J_1|+|J_2|\le 2^{n^q-i-1}+2^{n^q-i-1}=2\cdot2^{n^q-i-1}=2^{n^q-i},
$$
 so $J\in \{1,2,\dots,k\}^{\le 2^{n^q-i}}$.
Since 
$b_t\in \sum_{j\in J_t} \grp{g_j}$   for $t=0,1$,
we have $b=b_0+b_1\in \sum_{j\in J} \grp{
g_j}$.
Since $z_t\in\Z^\a$ for $t=0,1$,
we have $z=z_0+z_1\in \Z^\a$.
Finally, since $k_0\in\Z$, we get $k_0 z+w \in\Z^\a$ and thus $s^p_i(k_0 z+ w) \in \Zet[s^p_i]^\a$.
Combining all this with
\eqref{eq:x}, we conclude that 
$$
x\in U^p_{i}+l (g_0+h)+\sum_{j\in J} \grp{g_j}+ \Zet[s^p_i]^\a
$$
with $l\in\Z$, $J\in\{1,2,\dots,k\}^{\le 2^{n^q-i}}$
and 
$|l|\le 2^{n^q-i}$.
This finishes the proof of the inclusion
\eqref{eq:inlcusion}.
\end{proof}

\begin{claim}
\label{claim:3}
$U_i^q\cap (\Q_{\pi^p}^\a \Oplus H)=U_i^p$ for every $i\in n^p+1$.
\end{claim}
\begin{proof} 
It follows from \eqref{eq:3} that
$U_i^p\subseteq U_i^q$.
Since $U_i^p\subseteq \Q_{\pi^p}^\a \Oplus H$ by (4$_p$), 
this establishes the inclusion
$U_i^p\subseteq U_i^q\cap (\Q_{\pi^p}^\a \Oplus H)$.

To prove the reverse inclusion, we fix an arbitrary $x\in U_i^q\cap (\Q^\a_{\pi^p} \Oplus H)$ and we are going to show that $x\in U_i^p$. 
Since $x\in U_i^q$, we can apply 
Claim \ref{claim:1} to fix
$y\in U_i^p$, $l\in\Z$, a finite set 
$J\subseteq \{1,2,\dots,k\}$, a family 
$\{m_j:j\in J\}\subseteq \Z$ and $z\in\Z^\a$ 
such that $|l|\le 2^{n^q-i}$, $|J|\le 2^{n^q-i}$ and 
\begin{equation}
\label{eq:23}
x=y+l (g_0+h)+\sum_{j\in J} m_j g_j+ s_i^p z.
\end{equation} 
Recall that $x\in \Q^\a_{\pi^p} \Oplus H$.
Furthermore,
$y\in U_i^p\subseteq \Q^\a_{\pi^p} \Oplus H$
by (4$_p$).
Finally,
$s_i^p z \in\Z^\a\subseteq \Q_{\pi^p}^\a$.
Since $\pi_0=\pi^p$ by \eqref{eq:pi:modified}
and $\Q_{\pi_0}^\a$ is a group, from \eqref{eq:23} we get 
\begin{equation}
\label{eq:99}
l (g_0+h)+\sum_{j\in J} m_j g_j=x-y-s_i^p z \in \Q^\a_{\pi_0} \Oplus H,
\end{equation}
so
\begin{equation}
\label{eq:45:f}
l g_0=-h-\sum_{j\in J} m_j g_j+\Q^\a_{\pi_0} \Oplus H
\in \grp{\{g_j:j\in J\}} + \Q^\a_{\pi_0} \Oplus H.
\end{equation}
Since $\{g_0\}\cup\{g_j:j\in J\}\subseteq \Q^\a$,
$\Q_{\pi_0}^\a\subseteq\Q^\a$, $\Q^\a$ is a group and 
the sum $\Q^\a+H=\Q^\a\oplus H$ is direct,
from the inclusion \eqref{eq:45:f} we obtain the stricter inclusion
$l g_0 \in \grp{\{g_j:j\in J\}} + \Q^\a_{\pi_0}$.

Since $|J|\le 2^{n^q-i}\le 2^{n^q}=2^{n^p}<k$
by \eqref{eq:k}, $J$ is a proper subset of the set $\{1,2,\dots,k\}$.
Now all the assumptions of Lemma \ref{iterative:lemma}
are satisfied (with $s=s_{n^p}^p$), so from item (B) of
this lemma we conclude that $l = 0$. 
From this and
\eqref{eq:99}, we obtain 
$\sum_{j\in J} m_j g_j\in \Q^\a_{\pi_0} \Oplus H$,
so we can fix $a\in \Q^\a_{\pi_0}$ and $b\in H$
satisfying
$\sum_{j\in J} m_j g_j=a+b$.
Since $\Q^\a_{\pi^q}$ is a subgroup of $\Q^\a$, from 
\eqref{all:gs} we get
$\sum_{j\in J} m_j g_j\in\Q^\a_{\pi^q}$.
Therefore,
$$
b=(-a)+\sum_{j\in J} m_j g_j\in \Q^\a_{\pi_0}+\Q^\a_{\pi^q}\subseteq \Q^\a.
$$
Since $b\in H$, we get $b\in \Q^\a\cap H$. Since the sum
$\Q^\a+ H=\Q^\a\oplus H$ is direct, we conclude that $b=0$.
Thus, $\sum_{j\in J} m_j g_j=a\in \Q^\a_{\pi_0}$.
Since 
$J\subseteq \{1,2,\dots,k\}$,
we get
$$
\sum_{j\in J} m_j g_j\in
\grp{\{g_j:j\in J\}}\cap \Q^\a_{\pi_0}
\subseteq 
\grp{\{g_1,g_2,\dots,g_k\}}\cap \Q^\a_{\pi_0}
\subseteq \Zet[s^p_{n^p}]^\a
$$ by 
subitem~(ii) of item~(A) of Lemma \ref{iterative:lemma} applied with $s=s_{n^p}^p$.
Combining this with \eqref{eq:23},
we get $x\in y+ \Zet[s^p_{n^p}]^\a + \Zet[s^p_i]^\a$.
Since $i\le n_p$, $s_i^p$ divides $s_{n^p}^p$ by (8$_p$), so
$\Zet[s^p_{n^p}]^\a \subseteq \Zet[s^p_i]^\a$.
This gives
$x\in y+ \Zet[s^p_i]^\a \subseteq U_i^p+ \Zet[s^p_i]^\a =U_i^p$
by
(6$_p$).
\end{proof}

\begin{claim}
\label{claim:5}
$q\le p$.
\end{claim}
\begin{proof}
We need to check conditions (i$_q^p$)--(iv$_q^p$).
(i$_q^p$) follows from \eqref{eq:pi}.
Since $n^q=n^p$ by our definition, (ii$_q^p$) holds.
(iii$_q^p$) is proved in Claim \ref{claim:3}.
(iv$_q^p$) holds by the definition of $s_i^q$.
\end{proof}

The density of $D_{g+h}$ in $(\P,\le)$ follows from Claims
\ref{claim:1a}, \ref{claim:2} and
\ref{claim:5}.
\end{proof}

\section{Main theorem}
\label{Sec:7}

We shall need the following folklore lemma.

\begin{lemma}
\label{countable:generic:filter}
If $\mathscr{D}$ is an at most countable family of dense subsets of a non-empty poset 
$(\P,\le)$, then there exists an at most countable subset $\F$ of $\P$
such that $(\F,\le)$ is a linearly ordered set and 
$\F\cap D\not=\emptyset$ for every $D\in\mathscr{D}$.
\end{lemma}
\begin{proof}
Since the family $\mathscr{D}$ is at most countable, we can fix an enumeration $\mathscr{D}=\{D_n:n\in\N^+\}$ of elements of $\mathscr{D}$. Since $\P\not=\emptyset$, there exists $p_0\in\P$. By induction  on $n\in\N^+$, we can choose 
$p_n\in D_n$ such that $p_n\le p_{n-1}$; this is possible because 
$D_n$ is dense in $(\P,\le)$. Now $\F=\{p_n:n\in\N^+\}$ is the desired subset of $\P$.
\end{proof}

The next
theorem provides
a positive answer to a more general version of 
\cite[Question 13.1]{DS_SSGP}.

\begin{theorem}
\label{the:theorem}
Suppose that $\a \in\N^+$ and $G$ is a wide subgroup of $\Q^\a$.
Then for each at most countable abelian group $H$, the 
direct sum 
$K = G 
\oplus
H$
admits a metric $\ssgp$ 
topology.
 \end{theorem} 
\begin{proof}

Fix $m\in\N^+$ and let $G$ be a wide subgroup of $\Q^\a$.
Let $H$ be a countable abelian group. We work in the direct sum 
$\Q^\a\oplus H$ and consider its subgroup 
$K=G+H$. Clearly,
the sum $G+H=G\oplus H$ is direct. Therefore,
$K$ can also be considered as the direct product $G\times H$ of $G$ and $H$. This observation means that it suffices to find the required group topology $\mathscr{T}$ on the subgroup $K=G+H$ of the direct sum $\Q^\a\oplus H$.

Let $(\P,\le)$ be the poset from Definition \ref{def:P}
which uses our $m\in\N^+$, $G$ and $H$ as its parameters.

As a subgroup of $\Q^\a$,
$G$ is at most countable.
Since $H$ is at most countable as well, so is the 
sum
$K=G+ H$.
Therefore, the family 
\begin{equation}
\label{family:D}
\mathscr{D}=
\{C_x: x\in K\setminus\{0\}\}
\cup
\{A_n\cap D_x:n\in\N, x\in K\}
\end{equation}
of subsets of $\P$ is at most countable.

Lst us check that
all members of $\mathscr{D}$ are dense in $(\P,\le)$. 
By Lemma \ref{trivial:dense:sets:iii},
each $C_x$ for $x\in K \setminus\{0\}$ is dense in 
$(\P,\le)$.
Let $n\in\N$ and $x \in K$.
Since $A_n$ is dense and downward-closed in $(\P,\le)$
by Lemma \ref{trivial:dense:sets:i}
and 
$D_x$ is dense in $(\P,\le)$ by Lemma \ref{ssgp:density}, 
from Lemma \ref{downwardclosed:dense:sets} we conclude that
$A_n\cap D_x$ is dense in $(\P,\le)$. 

Apply Lemma \ref{countable:generic:filter} to find a countable
set $\F\subseteq \P$ such that $(\F,\le)$ is a linearly ordered set and 
$\F\cap D\not=\emptyset$ for every $D\in\mathscr{D}$.
For every $i\in\N$, define
\begin{equation}
\label{def:U_i}
U_i=\bigcup\{U_i^p:p\in\F\text{ and }i\le n^p\}.
\end{equation}

Our nearest goal is to show that the family 
\begin{equation}
\label{eq:fam:B}
\B = \{U_i:i\in\N\}
\end{equation}
is a neighbourhood base at $0$ 
of a Hausdorff group topology $\mathscr{T}$ on $K$.
The verification of this will be split into three claims.

\begin{claim}
\label{claim:6}
$\bigcap_{n\in\N} U_n = \{0\}.$
\end{claim}

\begin{proof}
Let $n\in\N$ be arbitrary.
Since $A_n\cap D_0\in\mathscr{D}$ by \eqref{family:D}, there exists 
$p\in A_n\cap D_0\cap\F$.
Therefore, $n\le n^p$ by the definition of $A_n$.
Now $0\in U_n^p$ by the condition (4$_p$) imposed on $\P$. 
Since $p\in \F$ and $n\le n^p$, it follows from \eqref{def:U_i} that
$U_n^p\subseteq U_n$. This shows that $0\in U_n$.
Since $n\in\N$ was chosen arbitrarily, we conclude that 
$0\in \bigcap_{n\in\N} U_n$.

To prove the reverse inclusion 
$\bigcap_{n\in\N} U_n\subseteq \{0\}$, 
we choose $x \in K \setminus\{0\}$ arbitrarily
and show that $x \notin \bigcap_{n\in\N}U_n$.
Since $C_x \in\mathscr{D}$ by \eqref{family:D},
our choice of $\F$ allows us to fix $p\in C_x \cap\F$.
Then
\begin{equation}
\label{eq:35}
x \in (\Q_{\pi^p}^\a \Oplus H)\setminus U^p_{n^p}
\end{equation} 
by the definition of $C_x$.

Assume that $x \in U_{n^p}$. Thus, by \eqref{def:U_i},
there exists $q \in \F$ such that $n^p\le n^q$ and $x \in U^q_{n^p}$. 
Suppose that $p \leq q$.
Then $U^p_{n^p} \cap (\Q_{\pi^q}^\a \Oplus H) = U^q_{n^p}$ 
by (iii$_p^q$).
Since $x \in U^q_{n^p}$, this implies
$x \in U^p_{n^p}$, in contradiction with \eqref{eq:35}.
Similarly, assume that $q \leq p$.
Then 
$U^q_{n^p} \cap (\Q_{\pi^p}^\a \Oplus H) = U^p_{n^p}$
by (iii$_q^p$), so
$x \notin U^q_{n^p}$ by \eqref{eq:35}, in contradiction with 
$x \in U^q_{n^p}$.
This contradiction shows that neither $p\le q$ nor $q \leq p$ holds. 
Since $p,q\in \F$, this contradicts the fact that $\F$ is linearly ordered by $\le$.
This contradiction shows that our assumption that $x \in U_{n^p}$ is false, 
so $x \notin\bigcap_{n\in\N} U_n$.
\end{proof}

\begin{claim}
\label{claim:7}
$-U_{i} = U_{i}$ and $U_{i+1} + U_{i+1} \subseteq U_{i}$ for every $i \in \N$. 
\end{claim}
\begin{proof}
Let 
$x \in U_{i}$. 
By \eqref{def:U_i}, this means 
that $x \in U^p_{i}$ for some $p \in \F$ satisfying $i\le n^p$. By (5$_p$), we know that $-U^p_{i} = U^p_{i}$, so $-x \in U^p_{i}$. This proves the inclusion $-U_{i} \subseteq U_{i}.$ The reverse inclusion is proved analogously.

Finally, consider $x,y \in U_{i+1}$.
By \eqref{def:U_i}, there exist
$p,q \in \F$ such that $x \in U^p_{i+1}$, 
$y \in U^q_{i+1}$, $i+1\le n^p$ and $i+1\le n^q$. Since $(\F,\le)$ is a linear order, we may assume,  without loss of generality, that $q \leq p$. 
Since $i+1\le n^p$, 
(iii$_q^p$) implies
$x \in U_{i+1}^p\subseteq U_{i+1}^q$.
Therefore, $x,y \in U^q_{i+1}$. Since 
$i+1\le n^q$, we have
$x+y \in U^q_{i+1} + U^q_{i+1} \subseteq U^q_{i}$ by (7$_q$).
Then $x+y \in U_{i}$ by \eqref{def:U_i}. This shows that $U_{i+1} + U_{i+1} \subseteq U_{i}$, as desired.
\end{proof}

\begin{claim}
\label{claim:9}
The family 
$\B$ as in \eqref{eq:fam:B}
is a neighbourhood base at $0$ 
of some Hausdorff group topology $\mathscr{T}$ on 
$K$.
\end{claim}
\begin{proof}
It easily follows from Claims \ref{claim:6} and \ref{claim:7}
that $U_m\subseteq U_n$ whenever $n,m \in \N$ and $n \leq m$. 
Combined with \eqref{eq:fam:B}, this implies that $\B$ is a filter base. By \eqref{eq:fam:B} and Claim \ref{claim:7}, 
$\B$ has the following two properties.

\begin{itemize}
\item For every $U \in \B$ there exists $V \in \B$ such that $V+V \subseteq U$; and
\item For every $U \in \B$ there exists $V \in \B$ such that $-V \subseteq U$.
\end{itemize}

By \cite[Theorem 3.1.5]{1}, the family 
$$
\mathscr{T} = \{ O \subseteq K : \forall\ a \in O\ \exists\ U \in \B\ (a + U \subseteq O) \}
$$ 
is a topology on $K$ making it into a topological group such that the family $\B$ is a neighbourhood base at $0$ comprised of $\mathscr{T}$-neighbourhoods of $0$. 
It follows from 
Claim \ref{claim:6} and \cite[Theorem 4.1.1]{1} that
$\mathscr{T}$
is Hausdorff. 
\end{proof}

\begin{claim}
The topological group $(K,\mathscr{T})$ has the small subgroup generating property. 
\end{claim}
\begin{proof}
We are going to check that $(K,\mathscr{T})$ has the property (ii)
from Proposition \ref{SSGP:reformulation}.

Let 
$W$ be a neighbourhood of $0$ in $(K,\mathscr{T})$.
By Claim \ref{claim:9},
there exists $i\in\N$ such that $U_i\subseteq W$.

Take an arbitrary $x \in K$. 
Since $A_i\cap D_x \in\mathscr{D}$, there exists $q\in \F\cap A_i\cap D_x$. Since $q\in D_x$, we can find
$k\in\N$ such that 
\begin{equation}
\label{eq:39}
x \in U^q_{n_q} + \grp{\Cyc(U^q_{n_q})}_k.
\end{equation}
Since $q\in A_i$, we have
$i\le n^q$,
so $U_{n^q}^q\subseteq U_i^q$
by Remark \ref{subgroups:in:U_i^p}~(i).
Since $q\in\F$, from \eqref{def:U_i} we get
$U_i^q\subseteq U_i$,
so
$U_{n^q}^q\subseteq U_i\subseteq W$, which implies
$\Cyc(U_{n^q}^q)\subseteq \Cyc(W)$.
Therefore, 
\begin{equation}
\label{eq:40}
U^q_{n_q} + \grp{\Cyc(U^q_{n_q})}_k \subseteq W + \grp{\Cyc(W)}_k.
\end{equation}
Combining \eqref{eq:39} and \eqref{eq:40},
we conclude that
$x \in 
\bigcup_{k\in\N^+} W + \grp{\Cyc(W)}_k.$ 
Since this inclusion holds for every $x\in K$, we get
$K \subseteq \bigcup_{k\in\N^+} W + \grp{\Cyc(W)}_k$. The converse inclusion is clear.
\end{proof}

Since $(K,\mathscr{T})$ is Hausdorff and has a countable base at $0$ by Claim \ref{claim:9}, it is metrizable.
\end{proof}

\end{document}